\documentclass[11pt]{article}

\usepackage{currfile}

\usepackage{amsmath,amsthm,verbatim,amssymb,amsfonts,amscd,graphicx}

\usepackage{bm}
\usepackage{tikz}
\usetikzlibrary{arrows,shapes,positioning}
\usetikzlibrary{decorations.markings}
\tikzstyle arrowstyle=[scale=1]
\tikzstyle directed=[postaction={decorate,decoration={markings, mark=at position 0.75 with {\arrow[arrowstyle]{stealth}}}}]
\tikzstyle redirected=[postaction={decorate,decoration={markings, mark=at position 0.35 with {\arrow[arrowstyle]{stealth}}}}]

\newtheorem{theorem}{Theorem}[section]

\newtheorem{definition}[theorem]{Definition}

\newtheorem{lemma}[theorem]{Lemma}
\newtheorem{observation}[theorem]{Observation}
\newtheorem{claim}{Claim}
\newtheorem{proposition}[theorem]{Proposition}

\DeclareMathOperator{\supp}{{\rm supp}}

\newcommand{\JCTB}{{\it J. Combin. Theory Ser. B}}

\newcommand{\DM}{{\it Discrete Math.}}

\newcommand{\SIAMDM}{{\it SIAM J. Discrete Math.}}

\newcommand{\CJM}{{\it Canad. J. Math.}}

\newcommand{\EJC}{{\it European J. Combin.}}

\begin{document}
\title{Flows on   flow-admissible signed graphs }
\author{Matt DeVos\\
Department of Mathematics\\
Simon Fraser University, Burnaby, B.C., Canada V5A1S6\\
Email: mdevos@sfu.ca\\
Jiaao Li, You Lu, Rong Luo, Cun-Quan Zhang\thanks{This research project has been partially supported
 by an  NSA grant   H98230-14-1-0154, an NSF grant DMS-1264800}, Zhang Zhang
 \\
 Department of Mathematics\\
 West Virginia University\\
 Morgantown, WV 26505\\
Email:   \{joli,yolu1, rluo, cqzhang\}@math.wvu.edu, zazhang@mix.wvu.edu}

\date{}
\maketitle

\begin{abstract}
In 1983, Bouchet  proposed a conjecture that every flow-admissible signed graph admits
 a nowhere-zero $6$-flow. 
Bouchet himself proved that
such signed graphs admit  nowhere-zero $216$-flows  and Z\'yka further
proved that
such signed graphs admit nowhere-zero $30$-flows. In this paper we show that every  flow-admissible signed graph admits a nowhere-zero $11$-flow.
\noindent


\end{abstract}

\section{Introduction}
\label{sc: Introduction}

Graphs or signed graphs considered in this paper are finite and may have multiple edges or loops. For  terminology and notations not defined here we follow \cite{BM2008, Diestel2010, West1996}.

In 1983, Bouchet~\cite{Bouchet83} proposed a flow conjecture that
{\em every flow-admissible signed graph admits a nowhere-zero $6$-flow}.
 Bouchet~\cite{Bouchet83} himself proved that
such signed graphs admit  nowhere-zero $216$-flows;
Z\'yka~\cite{Zyka1987}
proved that
such signed graphs admit nowhere-zero $30$-flows.
In this paper,
we  prove the following result.

\begin{theorem}
\label{thm: 11-flow}
Every flow-admissible signed graph admits a nowhere-zero $11$-flow.
\end{theorem}

In fact,
 we prove a stronger and very structural result as follows, and
Theorem~\ref{thm: 11-flow} is an immediate corollary.

  \begin{theorem}
\label{TH: e=10}
Every flow-admissible signed graph $G$
admits   a $3$-flow $f_1$ and a $5$-flow $f_2$ such that $f = 3f_1 + f_2$
is a nowhere-zero $11$-flow,  $|f(e)| \not = 9$ for each edge $e$, and $|f(e)| = 10$ only if $e \in  B(\supp(f_1))\cap B(\supp(f_2))$, where $B(\supp(f_i))$ is the set of all bridges of the subgraph induced by the edges of
$\supp(f_i)$ $(i=1,2)$.
 \end{theorem}

Theorem~\ref{TH: e=10}
 may suggest an
  approach to further reduce  $11$-flows to $9$-flows.

 The main approach  to prove the $11$-flow theorem is the following result, which, we believe,  will be a powerful tool in the  study of integer flows of signed graphs, in particular to resolve Bouchet's $6$-flow conjecture.
\begin{theorem}
\label{thm: balanced-6-flow}
Every flow-admissible signed graph admits a balanced  nowhere-zero $\mathbb Z_{2} \times \mathbb Z_{3}$-flow.
\end{theorem}

A $\mathbb{Z}_2\times \mathbb{Z}_3$-flow $(f_1,f_2)$ is called {\em balanced} if $\supp (f_1)$ contains an even number of negative edges.

The rest of the paper is organized as follows: Basic notations and definitions will be introduced in Section~\ref{se:defintions}.  Section~\ref{se:mod flow} will discuss the conversion of modulo flows into integer flows. In particular a new result to convert a modulo $3$-flow to an integer $5$-flow will be introduced and its proof will be presented in Section~\ref{se:mod3-5}. The proofs of Theorems~\ref{TH: e=10} and \ref{thm: balanced-6-flow} will be presented in Sections ~\ref{se:proof 11 flow} and ~\ref{se:6-flow}, respectively.

\section{Signed graphs, switch operations, and flows}
\label{se:defintions}

Let $G$ be a graph.
For $U_1, U_2\subseteq V(G)$,  denote  by $\delta_{G}(U_1, U_2)$  the set of edges with one end in $U_1$ and the other in $U_2$.
For convenience, we write $\delta_G(U_1, V(G)\setminus U_1)$  and $\delta_G(\{v\})$ for  $\delta_G(U_1)$  and $\delta_G(v)$ respectively. The degree of $v$ is $d_G(v)=|\delta_G(v)|$. A $d$-vertex  is a vertex with degree $d$. Let $V_d(G)$ be the set of $d$-vertices in $G$. The maximum degree of $G$ is denoted by $\Delta(G)$. We use $B(G)$ to denote the set of cut-edges of $G$.

A \emph{signed graph} $(G,\sigma)$ is a graph $G$ together with a {\em signature} $\sigma: E(G) \to \{-1,1\}$. An edge $e \in E(G)$ is \emph{positive} if $\sigma(e) =1$ and \emph{negative} otherwise. Denote the set of all negative edges of $(G,\sigma)$ by $E_N(G,\sigma)$. For a vertex $v$ in $G$, we define a new signature $\sigma'$ by changing $\sigma'(e) = -\sigma(e)$ for each $e\in \delta_G(v)$.
We say that $\sigma'$ is obtained from $\sigma$ by making a \emph{switch} at the vertex $v$.
Two signatures are said to be {\em equivalent} if one can be obtained from the other by making a sequence of switch operations. Define the {\em negativeness} of $G$ by $\epsilon (G,\sigma)=\min\{|E_N(G,\sigma')| : \mbox{ $\sigma'$ is equivalent to $\sigma$}\}$. A signed graph is balanced if its negativeness is $0$. That is it is equivalent to a graph without negative edges. For a subgraph $G'$ of $G$, denote $\sigma(G') = \prod_{e\in E(G')} \sigma(e)$.

For convenience, the signature $\sigma$ is usually omitted if no confusion arises or is written as $\sigma_G$ if it needs to emphasize $G$. If there is no confusion from the context, we simply use $E_N(G)$ for  $E_N(G,\sigma)$ and use  $\epsilon (G)$ for  $\epsilon (G,\sigma)$.

Every edge of $G$ is composed of two half-edges $h$ and $\hat{h}$, each of which is incident with one end.
Denote the set of half-edges of $G$ by $H(G)$ and the set of half-edges incident with $v$ by $H_G(v)$.
For a half-edge $h \in H(G)$, we refer to $e_{h}$ as the edge containing $h$.
An {\em orientation} of a signed graph $(G, \sigma)$ is a mapping $\tau: H(G) \to \{-1,1\}$
such that $\tau(h) \tau(\hat{h}) = -\sigma(e_{h})$ for each $h \in H(G)$.
It is convenient to consider $\tau$ as an assignment of orientations on $H(G)$.
Namely, if $\tau(h) =1$, $h$ is a half-edge oriented away from its end and otherwise towards its end. Such an ordered triple $(G, \sigma, \tau)$ is called a {\em bidirected graph}.

\begin{definition}
Assume that $G$ is a signed graph associated with an orientation $\tau$. Let $A$ be an abelian group and $f: E(G) \to A$ be a mapping. The {\em boundary} of $f$ at a vertex $v$ is defined as
$$\partial f (v)=\sum_{h\in H_G(v)} \tau (h) f(e_h).$$
The pair $(\tau,f)$ (or simplify, $f$) is an \emph{$A$-flow} of $G$ if $\partial f (v)=0$ for each $v\in V(G)$, and is an (integer) \emph{$k$-flow} if it is a $\mathbb{Z}$-flow and $|f(e)|<k$ for each $e\in E(G)$.
\end{definition}

Let $f$ be a flow of a signed graph $G$. The {\rm support} of $f$, denoted by $\supp (f)$, is the set of edges $e$ with $f(e)\neq 0$. The flow $f$ is {\em nowhere-zero} if $\supp (f) = E(G)$. For convenience, we abbreviate the notions of
{\em nowhere-zero $A$-flow} and
{\em nowhere-zero $k$-flow} as
{\em
$A$-NZF}
 and
 {\em $k$-NZF}, respectively.
 Observe that $G$ admits an $A$-NZF (resp., a $k$-NZF) under an orientation $\tau$ if and only if it admits an $A$-NZF (resp., a $k$-NZF) under any orientation $\tau'$.
A $\mathbb Z_k$-flow is also called a modulo $k$-flow. For an integer flow $f$ of $G$ and a positive integer $t$, let $E_{f=\pm t} :=\{e\in E(G) : |f(e)|=t\}$.

 A signed graph $G$ is {\em flow-admissible} if it admits a $k$-NZF for some  positive integer $k$. Bouchet~\cite{Bouchet83} characterized all flow-admissible signed graphs as follows.

  \begin{proposition} {\rm (\cite{Bouchet83})}\label{flow admissible}
 A connected signed graph $G$ is flow-admissible if and only if $\epsilon(G)\neq 1$ and there is no cut-edge $b$ such that $G-b$ has a balanced component.
 \end{proposition}

\section{Modulo flows on signed graphs}
\label{se:mod flow}

Just like in the study of flows of ordinary graphs and as Theorem~\ref{thm: balanced-6-flow} indicates, 
 the key to make further improvement and to eventually solve  Bouchet's $6$-flow conjecture is to further study how to convert modulo $2$-flows  and modulo $3$-flows into integer flows. The following lemma  converts  a modulo $2$-flow into an integer $3$-flow.

\begin{lemma}
[\cite{CLLZ2018}]
\label{TH: 2-to-3}
If a signed graph is connected and admits a  $\mathbb{Z}_2$-flow $f_1$
such that $\supp(f_1)$ contains an even number of negative edges,
then it also admits a $3$-flow $f_2$ such that
 $\supp(f_1)\subseteq \supp(f_2)$ and $|f_2(e)| = 2$ if and only if $e \in B(\supp(f_2))$.
\end{lemma}

In this paper, we  will show that  one can convert a   $\mathbb{Z}_3$-NZF to a  very special  $5$-NZF.

\begin{theorem}
\label{lm: 1,2,4-flow}
Let $G$ be a signed graph admitting a $\mathbb Z_{3}$-NZF.
Then $G$ admits a $5$-NZF $g$ such that  $E_{g=\pm3}=\emptyset$  and $E_{g=\pm4}\subseteq B(G)$.
\end{theorem}

 Theorem~\ref{lm: 1,2,4-flow}  is  also a key tool in the proof  of the $11$-theorem and its proof will be presented in  Section~\ref{se:mod3-5}.

\begin{remark}
Theorem~\ref{lm: 1,2,4-flow} is sharp in the sense that  there is an infinite family of signed graphs that admits a $\mathbb Z_{3}$-NZF but does not admit a $4$-NZF. For example, the signed graph obtained from a tree in which each vertex is of degree one or three by adding a negative loop at each vertex of degree one.  An illustration is shown in  Fig.~\ref{FIG: no-3-flow}.
\end{remark}

\begin{figure}[h]
\begin{center}
\begin{tikzpicture}[scale=1.2]
\path(-0.8,0.5) coordinate (u1);\draw [fill=black] (u1) circle (0.08cm);
\path (0.8,0.5) coordinate (u2);\draw [fill=black] (u2) circle (0.08cm);
\path (-2.4,0.5) coordinate (u3);\draw [fill=black] (u3) circle (0.08cm);
\path (2.4,0.5) coordinate (u4);\draw [fill=black] (u4) circle (0.08cm);
\path (-0.8,-0.5) coordinate (u5);\draw [fill=black] (u5) circle (0.08cm);
\path (0.8,-0.5) coordinate (u6);\draw [fill=black] (u6) circle (0.08cm);

\draw [redirected, directed, line width=0.85] (u1)--(u3);
\draw [directed, redirected, line width=0.85] (u1)--(u2);
\draw [directed, redirected, line width=0.85] (u4)--(u2);
\draw [directed, redirected, line width=0.85] (u1)--(u5);
\draw [directed, redirected, line width=0.85] (u6)--(u2);


\path (-0.8,-1.5) coordinate (v1); \draw [directed, dotted, line width=1] (v1) arc (270:90:0.5cm);
\path (-0.8,-1.5) coordinate (v11); \draw [directed, dotted, line width=1] (v11) arc (-90:90:0.5cm);

\path (0.8,-0.5) coordinate (v2); \draw [redirected, dotted, line width=1] (v2) arc (90:270:0.5cm);
\path (0.8,-0.5) coordinate (v21); \draw [redirected, dotted, line width=1] (v21) arc (90:-90:0.5cm);

\path (2.4,0.5) coordinate (v3); \draw [redirected, dotted, line width=1] (v3) arc (180:0:0.5cm);
\path (2.4,0.5) coordinate (v31); \draw [redirected, dotted, line width=1] (v31) arc (180:360:0.5cm);

\path (-3.4,0.5) coordinate (v4); \draw [directed, dotted, line width=1] (v4) arc (180:0:0.5cm);
\path (-3.4,0.5) coordinate (v41); \draw [directed, dotted, line width=1] (v41) arc (180:360:0.5cm);

\end{tikzpicture}
\end{center}
\caption{\small A signed graph admitting a $\mathbb{Z}_3$-NZF with all edges assigned with $1$, but no $4$-NZF.}
\label{FIG: no-3-flow}
\end{figure}
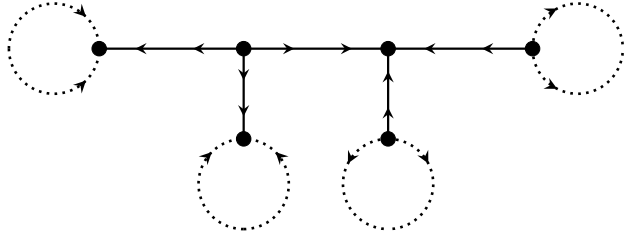

\section{Proof of the $11$-flow theorem}
\label{se:proof 11 flow}
Now we are ready to prove  Theorem~\ref{TH: e=10}.

\medskip
\noindent {\bf Proof of Theorem~\ref{TH: e=10}.} Let $G$ be a connected flow-admissible signed graph. By Theorem~\ref{thm: balanced-6-flow},  $G$ admits  a balanced $\mathbb Z_{2} \times \mathbb Z_{3}$-NZF $(g_1, g_2)$.
By  Lemma~\ref{TH: 2-to-3}, $G$ admits a $3$-flow $f_1$ such that $\supp(g_1)\subseteq \supp(f_1)$ and $|f_1(e)|=2$ if and only if $e \in B(\supp(f_1))$.

By  Theorem~\ref{lm: 1,2,4-flow}, $G$ admits a $5$-flow $f_2$ such that $\supp(f_2)=\supp(g_2)$ and
\begin{equation}
E_{f_2=\pm 3}=\emptyset.
\label{EQ: no 3}
\end{equation}

Since  $(g_1,g_2)$ is a $\mathbb{Z}_2\times \mathbb{Z}_3$-NZF of $G$, \begin{equation}
{\rm supp}(f_1)\cup{\rm supp}(f_2)=\supp (g_1)\cup \supp (g_2) = E(G).
\label{EQ: supp}
\end{equation}

We are to show that
 $f= 3f_1 + f_2$ is a nowhere-zero  $11$-flow described in the theorem.
  Since $|f_1(e)|\leq 2$ and $|f_2(e)| \leq 4$, we have 
$$|f(e)| = |(3f_1+f_2)(e)| \leq 3|f_1(e)| + |f_2(e)| \leq 10
~~~~ \forall e\in E(G).$$
Furthermore,
 by applying Equations~(\ref{EQ: no 3}) and
(\ref{EQ: supp}),
$$3f_1(e)+ f_2(e) \neq 0, \pm 9 ~~~~ \forall e\in E(G).$$
 If $|f(e)| = 10$ for some edge $e \in E(G)$,
 then    $|f_1(e)| = 2$ and $|f_2(e)| = 4$.
Thus, by Lemmas~\ref{TH: 2-to-3}
and~\ref{lm: 1,2,4-flow} again,
the edge
 $e \in B(\supp(f_1))\cap B(\supp(f_2))$ and hence $f = 3f_1 + f_2$ is the $11$-NZF described in Theorem~\ref{TH: e=10}.
  \hfill $\Box$

\section{Proof of Theorem~\ref{lm: 1,2,4-flow}}
\label{se:mod3-5}
 As the preparation of the proof of Theorem~\ref{lm: 1,2,4-flow}, we first need some necessary lemmas.

 The first lemma is a stronger form of the famous Petersen's theorem, and here we omit its proof (see Exercise 16.4.8 in \cite{BM2008}).

\begin{lemma}
\label{pp: matching1}
  Let $G$ be a bridgeless cubic graph and $e_0\in E(G)$. Then $G$ has two perfect matchings $M_1$ and $M_2$ such that $e_0\in M_1$ and $e_0\notin M_2$.
\end{lemma}

 We also need a splitting lemma due to Fleischner \cite{Fleischner1976}.

Let $G$ be a graph and $v$ be a vertex.
If $F\subset \delta_{G}(v)$, we denote by $G_{[v;F]}$ the graph obtained from $G$
by splitting the edges of $F$ away from $v$.
That is, adding a new vertex $v^{*}$ and changing the common end of edges in $F$ from $v$ to $v^{*}$.

\begin{lemma}\label{LE: splitting}
{\rm(\cite{Fleischner1976})}
Let $G$ be a bridgeless graph and $v$ be a vertex.
If $d_{G}(v)\geq 4$ and $e_{0}, e_{1},e_{2}\in \delta_{G}(v)$
are chosen in a way that $e_{0}$ and $e_{2}$
are in different blocks when $v$ is a cut-vertex,
then either $G_{[v;\{e_{0}, e_{1}\}]}$ or $G_{[v;\{e_{0}, e_{2}\}]}$ is bridgeless.
Furthermore, $G_{[v;\{e_{0}, e_{2}\}]}$ is bridgeless if $v$ is a cut-vertex.
\end{lemma}

Let $G$ be a signed graph. A path $P$ in $G$  is called
 a {\em subdivided edge}
of $G$ if every internal vertex of $P$ is  a $2$-vertex. The {\em suppressed graph} of $G$, denoted by $\overline{G}$, is the
signed graph obtained from $G$ by replacing each maximal subdivided edge $P$ with
a single edge $e$ and assigning $\sigma(e)=\sigma(P)$.

The following result is proved in ~\cite{Xu2005}  which gives  a sufficient condition  when a modulo $3$-flow and an integer $3$-flow are equivalent for signed graphs.
\begin{lemma}
[\cite{Xu2005}]
\label{TH: Xu-Zhang-1}
Let $G$ be a bridgeless signed graph. If $G$ admits a $\mathbb{Z}_3$-NZF, then it also admits a $3$-NZF.
\end{lemma}

Lemma~\ref{TH: Xu-Zhang-1}  is strengthened  in the following lemma, which will be served as the induction base  in the proof of Theorem~\ref{lm: 1,2,4-flow}.
\begin{lemma}
\label{lm: mod-3-flow}
Let $G$ be a bridgeless signed graph admitting a $\mathbb Z_{3}$-NZF. Then for any $e_0\in E(G)$ and for any $i\in \{1, 2\}$, $G$ admits a $3$-NZF such that $e_0$ has the flow value $i$.
\end{lemma}

\begin{proof} Let $G$ be a counterexample with $\beta(G):=\sum_{v\in V(G)}|d_G(v)-2.5|$ minimum.
Since
 $G$ admits a $\mathbb{Z}_3$-NZF, there is an orientation $\tau$ of $G$ such that  for each $v\in V(G)$,
 \begin{equation}\label{eq-1}
 \partial \tau(v) :=\sum_{h\in H_G(v)}\tau(h) \equiv 0\pmod 3.
 \end{equation}

We claim $\Delta(G)\leq 3$. Suppose to the contrary that $G$ has a vertex $v$ with $d_G(v)\ge 4$.
By Lemma~\ref{LE: splitting}, we can split a pair of edges $\{e_1,e_2\}$ from $v$ such that
  the new signed graph $G'=G_{[v; \{e_1,e_2\}]}$ is still bridgeless. In $G'$, we consider $\tau$ as an orientation on $E(G')$ and
denote the common end of $e_{1}$ and $e_{2}$ by $v^{*}$.
If $\partial \tau(v^*)=0$, then $\beta(G')<\beta(G)$ and by Eq.~(\ref{eq-1}),  $\partial \tau(u)\equiv 0 \pmod 3$ for each $u\in V(G')$, a contradiction to the minimality of $\beta(G)$.
If $\partial \tau(v^*)\neq 0$, then we further add a positive edge $vv^{*}$ to $G'$ and denote the resulting signed graph by $G''$. Let $\tau''$ be the orientation of $G''$ obtained from $\tau$ by assigning
$vv^*$ with a direction such that $\partial {\tau''}(v^*)\equiv 0 \pmod 3$. Then by Eq.~(\ref{eq-1}),  $\partial {\tau''}(u)\equiv 0 \pmod 3$ for each $u\in V(G'')$. Since $\beta(G'')<\beta(G)$, we obtain a contradiction to the minimality of $\beta(G)$ again. Therefore $\Delta (G) \leq 3$.

Since $G$ is bridgeless,  every vertex of $G$ is of degree $2$ or $3$. Note that the existence of the desired $3$-NZFs is preserved under the suppressing operation. Then the suppressed signed graph $\overline{G}$ of $G$ is also a counterexample,  and $\beta(\overline{G})<\beta(G)$ when $G$ has some  $2$-vertices. Therefore $G$ is cubic by the minimality of $\beta(G)$.

Since $G$ is cubic, by Eq. (\ref{eq-1}), either $\partial \tau (v)=d_G(v)$ or $\partial \tau (v)=-d_G(v)$
  for each $v\in V(G)$. By Lemma~\ref{pp: matching1}, we can choose two perfect matchings $M_1$ and $M_2$ such that $e_0\notin M_1$ and $e_0\in M_2$.
  For $i=1, 2$, let $\tau_i$ be the orientation of $G$ obtained from $\tau$ by reversing the directions of all edges of $M_i$, and define a mapping $f_i: E(G) \to \{1,2\}$ by setting $f_i(e)=2$ if $e\in M_i$ and $f_i(e)=1$ if $e\notin M_i$. Then $f_1$ and $f_2$ are two desired $3$-NZFs of $G$ under $\tau_1$ and $\tau_2$, respectively,  a contradiction.
\end{proof}

Now we are ready to complete the proof of Theorem~\ref{lm: 1,2,4-flow}.

\medskip
\noindent
{\bf Proof of Theorem~\ref{lm: 1,2,4-flow}}
We will prove by induction on $t=|B(G)|$, the number of cut-edges in $G$. If $t=0$, then $G$ is bridgeless and  it is a direct corollary of Lemma~\ref{lm: mod-3-flow}. This establishes the base of the induction.

Assume  $t>0$. Let $e=v_1 v_2$ be a cut-edge in $B(G)$ such that one component, say $B_1$, of $G-e$
is minimal. Let $B_2$ be the other component of $G-e$. Since $G$ admits a $\mathbb{Z}_3$-NZF, $\delta(G) \geq 2$. Thus  $B_1$ is bridgeless and nontrivial.   WLOG assume $v_i\in B_i$ ($i=1,2$).
Let $B_i'$ be the graph obtained from $B_i$ by adding a negative loop $e_i$ at $v_i$. Then $B_i'$ admits a $\mathbb{Z}_3$-NZF since $G$ admits a $\mathbb{Z}_3$-NZF.
By induction hypothesis, $B_2'$ admits a $5$-NZF $g_2$ with $g_2(e_2)=a\in \{1, 2\}$. By Lemma~\ref{lm: mod-3-flow}, $B_1'$ admits a $3$-NZF
$g_1$ such that $g_1(e_1)=a$. Hence we can extend
$g_1$ and $g_2$ to a $5$-NZF $g$ of $G$ by setting $g(e)=2a$. Clearly $g$ is a desired $5$-NZF of $G$.
\hfill $\Box$

\section{Proof of Theorem~\ref{thm: balanced-6-flow}}
\label{se:6-flow}

In this section, we will complete the proof of
Theorem~\ref{thm: balanced-6-flow}, which is divided into two steps:  first to reduce it from general flow-admissible signed graphs to cubic shrubberies (see  Lemma~\ref{lm: reduction}); and then prove that every cubic shrubbery admits a balanced $\mathbb{Z}_2\times \mathbb{Z}_3$-NZF by showing a stronger result (see Lemma \ref{lm: water}).

We first need
some terminology and notations.
Let $G$ be a graph.
 For an edge  $e\in E(G)$, {\em contracting} $e$ is done by deleting $e$ and then (if $e$ is not a loop) identifying its ends.  For $S\subseteq E(G)$, we use $G/S$ to denote the resulting graph obtained from $G$ by contracting all edges in $S$.

For a path $P$, let $End(P)$ and $Int(P)$ be the sets of the ends and internal vertices of $P$, respectively.  For $U_1, U_2\subseteq V(G)$, a $(U_1, U_2)$-path is a path $P$ satisfying $|End(P)\cap U_i|=1$ and $Int(P)\cap U_i=\emptyset$ for $i=1,2$; if $G_1$ and $G_2$ are subgraphs of  $G$, we write $(G_1,G_2)$-path instead of $(V (G_1),V(G_2))$-path.
Let $C=v_1\cdots v_r v_1$ be a circuit. A {\em segment} of $C$ is the path $v_i v_{i+1}\cdots v_{j-1}v_j \pmod r$ contained in $C$ and is denoted by $v_iCv_j$ or $v_jC^-v_i$.  An $\ell$-circuit  is a circuit with length $\ell$.

For a plane graph $G$ embedded in the plane $\Pi$, a {\em face} of $G$ is a connected topological region (an open set) of $\Pi\setminus G$. If the boundary of a face is a circuit of $G$, it is called a {\em facial circuit} of $G$.  Denote $[1,k] = \{1,2,\dots,k\}$.

\subsection{Shrubberies}

Let $G$ be a signed graph and $H$ be a connected signed subgraph of $G$. An edge $e\in E(G)\setminus E(H)$ is called a {\em chord} of $H$ if both ends of $e$ are in $V(H)$. We denote the set of chords of $H$ by ${\cal C}_G(H)$ or simply ${\cal C}(H)$, and partition ${\cal C}(H)$ into
$$\mathcal{U}(H)=\mathcal{U}_G(H)=\{e\in \mathcal{C}(H) : H+e \mbox{ is unbalanced}\} \mbox{ and } \mathcal{B}(H)=\mathcal{B}_G(H)=\mathcal{C}(H)\setminus \mathcal{U}(H).$$
In particular, if $H$ is a circuit $C$ that either is unbalanced or satisfies $|\mathcal{U}(C)|+|V_2(G)\cap V(C)|\geq 2$, then it is \emph{removable}.

A signed graph $G$ is called a \emph{shrubbery} if it satisfies the following requirements:
\begin{itemize}
\item[(S1)] $\Delta (G)\leq 3$;

\item[(S2)] every signed cubic subgraph of $G$ is flow-admissible;

\item[(S3)] $|\delta_G(V(H))|+\sum_{x\in V(H)}(3-d_G(x))+2|\mathcal{U}(H)|\geq 4$ for any balanced and connected signed subgraph $H$ with $|V(H)|\geq 2$;

\item[(S4)] $G$ has no balanced $4$-circuits.
\end{itemize}

By the above definition, the following result is straightforward.

\begin{proposition}
\label{PROP: inheritance}
Every signed subgraph of a shrubbery is still a shrubbery.
\end{proposition}
\begin{proof}
 Let $G'$ be an arbitrary signed subgraph of $G$. Obviously, $G'$ satisfies (S1), (S2) and (S4). We will show that $G'$ satisfies (S3). 
 
 Let $H$ be a balanced and connected signed subgraph of $G'$ with $|V(H)|\ge 2$.
Let $A_1=\delta_G(V(H))\setminus \delta_{G'}(V(H))$ and $A_2={\cal C}_G(H)\setminus {\cal C}_{G'}(H)$. Then
$$\sum_{x\in V(H)}(3-d_{G'}(x))-\sum_{x\in V(H)}(3-d_{G}(x))=|A_1|+2|A_2|.$$ Since   ${\cal U}_{G'}(H) \subseteq {\cal U}_G(H)$,  we have
$$|{\cal U}_G(H)|-|{\cal U}_{G'}(H)|\leq |A_2|.$$
Since $G$ is a shrubbery, 
 $$|\delta_{G'}(V(H))|+\sum_{x\in V(H)}(3-d_{G'}(x))+2|\mathcal{U}_{G'}(H)|\ge |\delta_G(V(H))|+\sum_{x\in V(H)}(3-d_G(x))+2|\mathcal{U}_G(H)|\ge 4.$$
 
 Therefore $G'$ satisfies (S3) and thus is a shrubbery.
  \end{proof}

Proposition~\ref{PROP: inheritance} will be applied frequently in the proof of Lemma~\ref{lm: water} and thus it will not be referenced explicitly.

The following two theorems and Lemma \ref{lm: flow extention} will be applied to reduce Theorem \ref{thm: balanced-6-flow}.

\begin{theorem}{\rm (\cite{Seymour81})}\label{6-flow}
Every ordinary bridgeless graph admits a $6$-NZF.
\end{theorem}

\begin{theorem}{\rm (\cite{Tutte54})}\label{group-flow}
Let $A$ be an abelian group of order $k$. Then an ordinary graph admits a $k$-NZF if and only if it admits an $A$-NZF.
\end{theorem}

Let $G$ be an ordinary oriented graph,  $T\subseteq E(G)$ and $A$ be an abelian group.
For any function $\gamma: T\to A$,  let ${\cal F}_{\gamma}(G)$ denote the number of $A$-NZF $\phi$
of G with $\phi(e)=\gamma(e)$ for every $e\in T$. For every $X\subseteq V(G)$, let $\alpha_X: E(G)\to \{-1,0,1\}$ be
given by the rule
$$
\alpha_X(e)=\left\{
\begin{array}{rl}
1 & \mbox{if $e\in \delta_G(X)$ is directed toward $X$}\\
-1 & \mbox{if $e\in \delta_G(X)$ is directed away $X$}\\
0 & \mbox{otherwise}.
\end{array}
\right.
$$
For any two functions $\gamma_1, \gamma_2$ from $T$ to $A$, we call $\gamma_1, \gamma_2$ {\em similar} if for every $X\subseteq V(G)$, the following holds
$$
\sum_{e\in T}\alpha_X(e)\gamma_1(e)=0 \mbox{ if and only if } \sum_{e\in T}\alpha_X(e)\gamma_2(e)=0.
$$

\begin{lemma}\label{Seymour}
(Seymour - Personal communication). Let $G$ be an ordinary oriented graph,  $T\subseteq E(G)$ and $A$ be an abelian group. If the two functions $\gamma_1, \gamma_2: T\to A$ are similar, then ${\cal F}_{\gamma_1}(G)={\cal F}_{\gamma_2}(G)$.
\end{lemma}

\begin{proof}
We proceed by induction on the number of edges in $E(G)\setminus T$. If this set is empty,
then ${\cal F}_{\gamma_i}(G)\leq 1$ and ${\cal F}_{\gamma_i}(G)=1$ if and only if $\gamma_i$ is an $A$-NZF of $G$ for $i=1,2$. Thus, the result
follows by the assumption. Otherwise, choose an edge $e\in E(G)\setminus T$. If $e$ is a cut-edge,  then
${\cal F}_{\gamma_i}(G)=0$ for $i=1,2$. If $e$ is a loop, then we have inductively that
$$
{\cal F}_{\gamma_1}(G)=(|A|-1){\cal F}_{\gamma_1}(G-e)=(|A|-1){\cal F}_{\gamma_2}(G-e)={\cal F}_{\gamma_2}(G).
$$
Otherwise, applying induction to $G-e$ and $G/e$ we have
$$
{\cal F}_{\gamma_1}(G)={\cal F}_{\gamma_1}(G/e)-{\cal F}_{\gamma_1}(G-e)={\cal F}_{\gamma_2}(G/e)-{\cal F}_{\gamma_2}(G-e)={\cal F}_{\gamma_2}(G).
$$
\end{proof}

The following lemma directly follows from Lemma \ref{Seymour}.

\begin{lemma}
\label{lm: flow extention}
  Let $G$ be an ordinary oriented graph and $A$ be  an abelian group. Assume that $G$ has an $A$-NZF. If $G$ has a
  vertex $v$ with $d_G(v)\leq 3$ and $\gamma: \delta_G(v)\to A \setminus \{0\}$ satisfies $\partial \gamma (v)=0$, then
  there exists an $A$-NZF $\phi$ such that $\phi|_{\delta_G(v)}=\gamma$.
\end{lemma}

\begin{proof}
Let $f$ be an $A$-NZF of $G$.  Since $d_G(v)\leq 3$,  $f|_{\delta_G(v)}$ is similar to $\gamma$. Thus by Lemma \ref{Seymour}, we have  ${\cal F}_{\gamma}(G)={\cal F}_{f|_{\delta_G(v)}}(G)\neq 0$. Therefore there exists an $A$-NZF $\phi$ such that $\phi|_{\delta_G(v)}=\gamma$.
\end{proof}

Now we can reduce Theorem \ref{thm: balanced-6-flow}.

\begin{lemma}
\label{lm: reduction}
The following two statements are equivalent.
\begin{itemize}
\item[{\rm (i)}] Every flow-admissible signed graph admits a balanced $\mathbb{Z}_2\times \mathbb{Z}_3$-NZF.

\item[{\rm (ii)}] Every cubic shrubbery admits a balanced $\mathbb{Z}_2\times \mathbb{Z}_3$-NZF.
\end{itemize}
\end{lemma}

\begin{proof}
(i) $\Rightarrow$ (ii). By (S2), every cubic shrubbery is flow-admissible, and thus (ii) follows  from (i).

(ii) $\Rightarrow$ (i).  Let $G$ be a counterexample to (i) with
 $\beta (G)=\sum_{v\in V(G)} |d_G(v)-2.5|$ minimum. Since $G$ is flow-admissible, it admits a $k$-NZF $(\tau,f)$ for some positive integer $k$ and thus $V_1(G)=\emptyset$.  Furthermore, by  the minimality of $\beta(G)$,  $G$ is connected  and $V_2(G)=\emptyset$  otherwise the suppressed signed graph $\overline{G}$ of $G$ is also flow-admissible and has smaller $\beta(\overline{G})$ than $\beta(G)$.
  We are going to show that $G$ is a cubic shrubbery  and thus  admits a balanced $\mathbb{Z}_2\times \mathbb{Z}_3$-NZF by (ii),  which is a contradiction to the fact that $G$ is a counterexample.   By the definition of shrubberies, we only need to prove (I)-(III) in the following.

 \medskip \noindent
  {\bf  I.} $G$ is cubic.

Suppose to the contrary that $G$ has a vertex $v$ with $d_G(v)\neq 3$. Then $d_G(v) \geq 4$.
Let  $\{e_1,e_2\} \subset \delta_G(v)$ and let $G'=G_{[v; \{e_1,e_2\}]}$.  Denote the new common end of $e_{1}$ and $e_{2}$ in $G'$ by $v^{*}$.
If $\partial f(v^*)=0$,  let $G''=G'$. If $\partial f(v^*)\neq 0$,  we further add a positive edge $vv^{*}$ with the direction from $v$ to $v^*$ and assign $vv^*$ with the weight $\partial f(v^*)$. Let $G''$ be the resulting signed graph. In both cases, $G''$ is flow-admissible and $\beta(G'')<\beta(G)$. By the minimality of $\beta(G)$, $G''$ admits a balanced $\mathbb{Z}_2\times \mathbb{Z}_3$-NZF, and thus so does $G$, a contradiction. This proves I.

  \medskip \noindent
  {\bf  II.} $|\delta_G(V(H))|+2|\mathcal{U}(H)|\geq 4$ for any balanced and connected signed subgraph $H$ with $|V(H)|\geq 2$.

Suppose to the contrary that $H$ is
 such a subgraph with $|\delta_G(V(H))|+2|\mathcal{U}(H)|\leq 3$. Let $X=V(H)$. Then $H'=G[X]-\mathcal{U}(H)$ is a balanced and connected signed subgraph of $G$. WLOG assume that all edges of $H'
$ are positive. Let $G_1=G/E(H')$. Then  $G_1$ is also flow-admissible.

Since $|\delta_G(X)|+2|\mathcal{U}(H)|\leq 3$, it follows from the choice of $G$ and Proposition~\ref{flow admissible} that  either $|\mathcal{U}(H)|=0$ and $ |\delta_G(X)|\in \{2, 3\}$ or $|\mathcal{U}(H)|=1$ and $|\delta_G(X)|=1$.  Let $x$ be the contracted vertex in $G_1$ corresponding to $E(H')$. Then $d_{G_1}(x)=|\delta_G(X)|+2|\mathcal{U}(H)|\in \{2, 3\}$ and $\beta(G_1)<\beta(G)$ since $|X|=|V(H)|\ge 2$.  By the minimality of $\beta(G)$, $G_1$ admits a balanced $\mathbb Z_{2} \times \mathbb Z_{3}$-NZF $(\tau_1, f_1)$, where $\tau_1$ is the restriction of $\tau$ on $G_1$.

 Let $H_X$ be the set of the half edges of each edge in $\delta_G(X)\cup \mathcal{U}(H)$ whose end is in $X$. Then $|H_X|=|\delta_G(X)|+2|\mathcal{U}(H)|=2$ or $3$. We add a new vertex $y$ to $H'+H_X$ such that $y$ is the common end of all $h\in H_X$, and denote the new graph by $G_2$. Since $G$ is flow-admissible, $G_2$ is a bridgeless ordinary graph and thus admits a balanced $\mathbb{Z}_2\times \mathbb{Z}_3$-NZF by Theorems \ref{6-flow} and \ref{group-flow}.
Let $\tau_2$ be the restriction of $\tau$ on $G_2$ and define $\gamma (h) = f_1(e_h)$ for each $h\in H_X$. Note that $\tau_2(h)=\tau_1(h)$ for each $h\in H_X$.  Since $(\tau_1,f_1)$ is a balanced $\mathbb{Z}_2\times \mathbb{Z}_3$-NZF of $G_1$, we have  $\partial \gamma (y) =-\partial f_1(x)= 0$. 
By Lemma~\ref{lm: flow extention}, there is a balanced $\mathbb{Z}_2\times \mathbb{Z}_3$-NZF $(\tau_2,f_2)$ of $G_2$ such that $f_2|_{\delta_{G_2}(y)}=\gamma=f_1|_{\delta_{G_1}(x)}$. Thus $(\tau_1,f_1)$ can be extended to a balanced $\mathbb{Z}_2\times \mathbb{Z}_3$-NZF of $G$, a contradiction.

This proves  (II).

\medskip \noindent
 {\bf  III.} $G$ has no balanced $4$-circuits.

Suppose to the contrary that $G$ has a balanced $4$-circuit $C$.  Then we may assume that all edges of $C$ are positive. Let $G'=G/E(C)$. Then $\beta(G')<\beta(G)$. By the minimality of $\beta(G)$, $G'$ admits a balanced $\mathbb Z_{2} \times \mathbb Z_{3}$-NZF, say $(f_1', f_2')$. Since $C$ is a circuit with all positive edges and $|E(C)| = 4$ and since $|\mathbb Z_{2} \times \mathbb Z_{3}|=6$, it is easy to extend  $(f_1', f_2')$ to a
  balanced $\mathbb{Z}_2\times \mathbb{Z}_3$-NZF  of $G$, a contradiction.
  This proves (III) and thus completes the proof of the lemma.
\end{proof}

\subsection{Nowhere-zero watering}

In this subsection, we will prove that every shrubbery admits a nowhere-zero watering (Lemma \ref{lm: water}). We need some preparations.

\begin{theorem}\label{th: 1-negative}
{\rm (\cite{MW1966})}
Let $G$ be a $2$-connected graph with $\Delta(G)\leq 3$ and let $y_1, y_2, y_3\in V(G)$. Then either there exists a circuit of $G$ containing $y_1, y_2, y_3$, or there is a partition of $V(G)$ into $\{X_1, X_2, Y_1, Y_2, Y_3\}$ with the following properties:
\begin{itemize}
\item[\rm (1)] $y_i\in Y_i$ for $i=1,2,3$;
\item[\rm (2)] $\delta_G(X_1, X_2)=\delta_G(Y_i, Y_j)=\emptyset$ for $1\leq i < j \leq 3$;
\item[\rm (3)] $|\delta_G(X_i, Y_j)|=1$ for $i=1,2$ and $j=1,2,3$.
\end{itemize}
\end{theorem}

Let $H$ be a contraction of $G$ and let $x\in V(G)$. We use $\hat{x}$ to denote the vertex in $H$ which $x$ is contracted into.

\begin{theorem}
\label{thm: linkage}
{\rm (\cite{Lu Luo Zhang})}
 Let $G$ be a 2-connected signed graph with  $|E_N(G)| = \epsilon(G)=k\ge 2$,
where $E_N(G)=\{x_1y_1, \dots, x_ky_k\}$. Then the following two statements are equivalent.
\begin{itemize}
\item[\rm (i)] $G$ contains no two edge-disjoint unbalanced circuits.

\item[\rm (ii)]  The graph $G$ can be contracted to a cubic
 graph $G'$
such that
either $G'-\{ \hat{x}_1 \hat{y}_1, \dots, \hat{x}_k\hat{y}_k \}$ is a $2k$-circuit  $C_1$ on the vertices $\hat{x}_1,  \dots, \hat{x}_k, \hat{y}_1, \dots, \hat{y}_k$ or  can be obtained from a $2$-connected cubic plane graph by selecting a facial circuit $C_2$ and inserting the  vertices
$\hat{x}_1,  \dots, \hat{x}_k, \hat{y}_1, \dots, \hat{y}_k$ on the edges of $C_2$ in such a way that
for every pair $\{ i, j \} \subseteq [1,k]$, the vertices $\hat{x_i}, \hat{x_j}, \hat{y_i}, \hat{y_j}$ are around the circuit $C_1$ or $C_2$ in this cyclic order.
\end{itemize}
\end{theorem}

\begin{lemma}
\label{lm: flow extention2}
{\rm (\cite{Jaeger1992})}
Let $G$ be an ordinary oriented graph and $A$ be  an abelian group. Then $G$ is connected if and only if for every function $\beta: V(G) \to A$ satisfying $\sum_{v\in V(G)} \beta(v) = 0$,  there exists $\phi: E(G) \to A$ such that $\partial \phi = \beta$.
\end{lemma}

Let $G$ be a signed graph with an orientation. A \emph{nowhere-zero watering} (briefly, NZW) of $G$ is a mapping   $f: E(G)\to \mathbb Z_{2} \times \mathbb Z_{3}-\{(0, 0)\}$ such that
\[
\partial f(v) = (0,0) ~\mbox{if} ~ d_G(v) = 3 ~\mbox{and}~ \partial f(v) = (0,\pm 1) ~\mbox{if}~ d_G(v)  = 1,2.
\]
Similar to flows, the existence of an NZW is also an invariant under switching operation.

\begin{lemma}
\label{lm: remove circuit}
Let $G$ be a shrubbery  and let $C$ be a removable circuit of $G$.  Then for every NZW $f'=(f_1',f_2')$ of  $G-V(C)$, there exists an NZW $f=(f_1,f_2)$ of $G$ so that $f(e) = f'(e)$ for every $e\in E(G')$ and  $\supp(f_1)=\supp(f_1')\cup E(C)$.
\end{lemma}
\begin{proof}

  We first  extend $f'$ to $f: E(G)\to \mathbb Z_{2} \times \mathbb Z_{3}$ as follows where $\alpha_e$ is a variable in $\mathbb{Z}_3$ for every $e \in \mathcal{U}(C)$.
  $$f(e) = \left\{
\begin{array}{ll}
(0,\pm 1) &~\mbox{if $e \in \delta(V(C))$}\\
(1,0)   &~\mbox{if $e\in E(C)$}\\
(0,1) &~\mbox{if $e\in \mathcal{B}(C)$}\\
(0,\alpha_e) &~\mbox{if $e\in \mathcal{U}(C)$.}
\end{array}
\right.
$$
Since every $v \in V(G)\setminus V(C)$ adjacent to a vertex in $V(C)$ has degree less than three in $G'$, we may choose values $f(e)$ for each edge $e \in \delta(V(C)$ so that $f$ satisfies the boundary condition for a watering at every vertex in $V(G)\setminus V(C)$.  Obviously by the construction  $\partial f_1(v) = 0$ for every $v \in V(C)$. So we need only adjust $\partial f_2(v)$ for $v \in V(C)$ to obtain a watering.
We  distinguish the following two cases.

\medskip \noindent
Case 1:  $C$ is unbalanced.

In this case $\mathcal{B}(C) = \emptyset$.
Choose arbitrary $\pm 1$ assignments to the variables $\alpha_e$. Since $C$ is unbalanced, for every vertex $u\in V(C)$, there is a function $\eta^u: E(C)\to  \mathbb{Z}_3$ so that $\partial \eta_u(u)=1$ and $\partial \eta_u(v)=0$ for any $v\in V(C)\setminus \{u\}$.  Now we may adjust $f_2$ by adding a suitable combination of the $\eta^u$ functions so that $f$ is an NZW of $G$, as desired.

\medskip \noindent
 Case 2:   $C$ is balanced.

 WLOG we may assume  that every edge of $C$ is positive and every unbalanced chord is oriented so that each half edge is directed away from its end. In this case, each negative chord $e$ contributes $-2f_2(e) = \alpha_e$ to the sum $\sum_{v\in V(C)}\partial f_2(v)$. For every $v \in V(C)\cap V_2(G)$, let $\beta_v$ be a variable in $\mathbb{Z}_3$.
   Since $|\mathcal{U}(C)|+|V_2(G)\cap V(C)|\geq 2$, we can choose  $\pm 1$ assignments to all of the variables $\alpha_e$ and $\beta_v$ so that the following equation is satisfied:
   $$\sum_{v\in V(C)}\partial f_2(v) = \sum_{v\in V(C)\cap V_2(G)}\beta_v.$$
   By Lemma~\ref{lm: flow extention2},  we may choose a function $\phi: E(C) \rightarrow \mathbb{Z}_3$ so that
   $$\partial \phi(v)  = \left\{
\begin{array}{rl}
\beta_v - \partial f_2(v) &~\mbox{if $v\in V(C)\cap V_2(G)$}\\
-\partial f_2(v)  &~\mbox{if $v\in V(C)\setminus V_2(G)$}.
\end{array}
\right.
$$
Now modify $f$ by adding $\phi$ to $f_2$ and then $f$ is an NZW of $G$, as desired.
\end{proof}

 A \emph{theta} is a graph consisting of two distinct vertices and three internally disjoint paths between them. A theta is \emph{unbalanced} if it contains an unbalanced circuit.  By the definition, the following observation is straightforward.

\begin{observation}\label{ob: no theta}
Let $G$ be a signed graph containing no unbalanced thetas and  $\Delta(G) \leq 3$. Then for any unbalanced circuit $C$ and any $x\in V(G)\setminus V(C)$, $G$ contains no two internal disjoint $(x,C)$-paths.
\end{observation}

\begin{lemma}
\label{lm: water}
Every shrubbery has an NZW. Furthermore, if $G$ is a shrubbery with an unbalanced theta or a negative loop and $\varepsilon\in \{-1,1\}$,
then $G$ has an NZW $f = (f_1, f_2)$ such that $\sigma (\supp(f_1))=\varepsilon$.
\end{lemma}

\begin{proof}
Let $G$ be a minimum counterexample  with respect to $E(G)$. Then $G$ is connected.

\begin{claim}
\label{cl: 2-connect}
  $G$ is $2$-connected, and thus contains no loops.
\end{claim}

\noindent{\it Proof of Claim \ref{cl: 2-connect}.}
Suppose to the contrary that $G$ has a cut vertex. Since $\Delta(G)\leq 3$, $G$ contains a cut edge $e=v_1 v_2$. Let $G_i$ be the component of $G-e$ containing $v_i$.  By the minimality of $G$, each $G_i$ admits an NZW
  $f^i=(f_1^i, f_2^i)$, and $\partial f_2^i(v_i)\neq 0$ since $d_{G_i}(v_i)\leq 2$. Thus we can obtain an NZW $f=(f_1, f_2)$ of $G$  by
  setting $f(e)=(0, 1)$ and $f|_{E(G_i)}=f^i$ or $-f^i$ according to the orientation of $e$ and the values of $\partial f_2^1(v_1)$ and $\partial f_2^2(v_2)$. Further, if $G$ contains an unbalanced theta or a negative loop, so does one
  component of $G-e$, say $G_1$. By the minimality of $G$, we choose $f^1$ such that $\sigma (\supp(f_1^1))=\epsilon \cdot \sigma (\supp(f_1^2))$.
  Hence $\sigma (\supp(f_1))=\epsilon \cdot \sigma (\supp(f_1^2))\cdot \sigma (\supp(f_1^2))=\epsilon$, a contradiction.   \hfill $\Box$

\medskip

\begin{claim}
\label{cl: remove}
  $G$ has no  removable circuit $C$ with one of the following properties:

  {\rm (A)} $G-V(C)$ contains an unbalanced theta.

  {\rm (B)} $G-V(C)$ is balanced and $\sigma(C)=\epsilon$.
\end{claim}

\noindent{\it Proof of Claim \ref{cl: remove}.}
Suppose the claim is not true.
 By the minimality of $G$,  there exists an NZW $f'=(f_1', f_2')$ of $G-V(C)$ such that $\sigma (\supp(f_1'))=\epsilon \cdot \sigma(C)$ in Case (A) and $\sigma (\supp(f_1'))=1$ in Case (B). By Lemma~\ref{lm: remove circuit}, $f'$ can be extended to an
NZW $f=(f_1, f_2)$ of $G$ such that  $\supp(f_1)=\supp(f_1')\cup E(C)$. Obviously, $\sigma (\supp(f_1))=\sigma(\supp(f_1'))\cdot \sigma(C)=\epsilon$, a contradiction.
\hfill $\Box$

\begin{claim}
\label{cl: 2-edge-cut}
Let $X\subset V(G)$ such that  $|X| \geq 2$, $G[X]$ is balanced and $|\delta_G(X)|=2$. If $G-X$ either contains an unbalanced theta, or is balanced and contains a circuit, then $X\subseteq V_2(G)$ and $G[X]$ is a path.
\end{claim}

\noindent{\it Proof of Claim \ref{cl: 2-edge-cut}.}
Suppose the claim fails. Let $X\subset V(G)$ be a minimal set with the above properties.
Recall that $G$ is $2$-connected by Claim \ref{cl: 2-connect}. Since $|\delta_G(X)|=2$, $G[X]$ is connected. If $G[X]$ is a path, then $X\subseteq V_2(G)$. Thus $G[X]$ is not a path. Since $G[X]$ is connected, we have $X\cap V_3(G) \not = \emptyset$.  Hence  $X$ is nontrivial and  $G[X]$ is $2$-connected by the minimality of $X$. By (S3),
 $X$ contains two vertices of $V_2(G)$. Let $C$ be a circuit in  $G[X]$ containing at least two $2$-vertices. Then $C$ is  removable and thus by Claim \ref{cl: remove}-(A), $G-X$ contains no unbalanced theta.   By the hypothesis, $G-X$ is balanced and contains a circuit.

 Denote $\delta_G(X) = \{e_1, e_2\}$. Since both $G[X]$ and $G-X$ are balanced, by possibly replacing $\sigma_G$ by an equivalent signature, we may assume that $\sigma_G(e_1) \in \{-1,1\}$ and that $\sigma_G(e) = 1$ for every other edge $e \in E(G)$. Obviously, if $\sigma_G(e_1)=1$ then $G$ is an ordinary graph and so we get a contradiction to Claim \ref{cl: remove}-(B) since $C$ is removable and balanced. Hence $\sigma_G(e_1)=-1$ and $e_1$ is the only negative edge in $G$.

Let $C'$ be an unbalanced circuit, which contains $e_1$. Then $C'$ is removable, $G-V(C')$ is balanced, and $\sigma(C') = -1$. By Claim \ref{cl: remove}-(B), we have $\epsilon=1$. Since $C$ is removable and $\sigma_G(C)=1=\epsilon$, $G-V(C)$ is unbalanced by Claim \ref{cl: remove}-(B) again. We may choose $C'$ such that $V(C') \cap V(C) = \emptyset$. Note that $e_1$ is the unique negative edge of $G$.  $C'$ contains the edge cut $\{e_1, e_2\}$.  Let $x \in V(C')\cap X$ and $C''$ be a circuit in $G-X$. Then there are two internal disjoint $(x,C'')$-paths $P_1$ and $P_2$ in $G- V(C)$ such that $e_i\in P_i$ for $i = 1,2$. Then $P_1\cup P_2\cup C''$ is an unbalanced theta  in $G- V(C)$.  This is a contradiction to Claim \ref{cl: remove}-(A).
\hfill $\Box$

\begin{claim}\label{cl: 3-edge-cut}
Let $X\subset V(G)$ such that  $|X| \geq 2$, $G[X]$ is balanced and $|\delta_G(X)|\leq 3$. For any two distinct ends $x_1, x_2$ in $X$ of $\delta_G(X)$, there is an $(x_1,x_2)$-path in $G[X]$ containing at least one vertex in $V_2(G)$.
\end{claim}

\noindent{\it Proof of Claim \ref{cl: 3-edge-cut}.}
Let $x_1x_1', x_2x_2'\in \delta_G(X)$, and $B_i$ be the maximal $2$-connected subgraph of $G[X]$ containing $x_i$ for $i=1, 2$. Then every edge in $\delta_{G[X]}(V(B_i))$ is a bridge of $G[X]$,  so $|\delta_{G}(V(B_i))|\leq |\delta_G(X)|$ since $G$ is $2$-connected.

If $V(B_1)\cap V(B_2)\neq \emptyset$, then $|V(B_1)\cap V(B_2)|\ge 2$ since $\Delta(G)\leq 3$, and thus $B_1=B_2$ by their maximality.  By (S3),
 there is a vertex $y_1\in V(B_1)\cap V_2(G)$.  Since $B_1$ is $2$-connected, it has a $(y_1,x_1)$-path $P_1$ and a $(y_1,x_2)$-path $P_2$ that are internally disjoint. Thus $P_1\cup P_2$ is a desired path.

If $V(B_1)\cap V(B_2)= \emptyset$, then for some $i\in \{1, 2\}$, say $i=1$,  $|\delta_G(V(B_1))|=2$ since $|\delta_{G}(V(B_j))|\leq |\delta_G(X)|\leq 3$ for $j=1,2$. Let $y_2\in V(B_1)$ be the end of the unique edge in $\delta_G(V(B_1))\setminus \{x_1x_1'\}$ and $P_3$ be a $(y_2,x_2)$-path in $G[X]$. If $x_1\in V_2(G)$, then every $(x_1,x_2)$-path is a desired path. If $x_1\in V_3(G)$, then $|V(B_1)|\ge 2$ and thus $B_1$ has a vertex $y_3\in V_2(G)\setminus \{y_2\}$ by (S3). Since $B_1$ is $2$-connected, it has an $(y_3,x_1)$-path $P_4$ and a $(y_3,y_2)$-path $P_5$  which are internally disjoint. Thus $P_3\cup P_4\cup P_5$ is a desired path.
\hfill $\Box$

\begin{claim}
\label{cl:3-vertices}
$G$ contains no two disjoint unbalanced circuits $C_1$ and $C_2$ such that $V_3 \subseteq V(C_1)\cup V(C_2)$.
\end{claim}

\noindent{\it Proof of Claim \ref{cl:3-vertices}.}
Suppose the claim fails. Let $C_1$ and $C_2$ be two disjoint unbalanced circuits such that $V_3 \subseteq V(C_1)\cup V(C_2)$.
Then every vertex of $G'= G-E(C_1\cup C_2)$ is of degree at most $2$. By   Claim \ref{cl: remove}-(A),  $G - V(C_i)$ contains no unbalanced theta for each $i= 1,2$. Thus every nontrivial component of $G'$ is a path with one end in $V(C_1)$ and the other end in $V(C_2)$.   Since $G$ is $2$-connected and $\Delta(G) \leq 3$, there are at least two  $3$-vertices in each $C_i$.

When $\epsilon=-1$, choose $x_1, x_2$ from $V_3(G)\cap V(C_1)$ such that the segment $P=x_1C_1x_2$ contains all vertices of $V_3(G)\cap V(C_1)$. Let $P_i$ be the path in $G'$ with one end $x_i$ and $y_i$ be the other end of $P_i$ for $i = 1,2$.  Since $C_2$ is unbalanced, there is a segment, say $y_1C_2y_2$, of $C_2$ such that the circuit $C=P\cup P_1\cup P_2\cup y_1C_2y_2$ is unbalanced, and thus $C$ is removable. This contradicts Claim \ref{cl: remove}-(B) since $G-V(C)$ is a forest (which is balanced).

When $\epsilon=1$,  by the minimality of $G$ and since  $G''=G-V(C_1\cup C_2)$  is a forest, $G''$ admits an NZW $f'= (f_1', f_2')$ with $\supp(f_1')=\emptyset$. By applying Lemma~\ref{lm: remove circuit} twice, we extend $f' = (f_1', f_2')$ to an NZW $f = (f_1, f_2)$ of $G$ such that $\supp(f_1)=E(C_1)\cup E(C_2)$. So $\sigma(\supp(f_1)) =\sigma (C_1) \cdot \sigma (C_2)=1$, a contradiction.

\begin{claim}
\label{cl: 2-disjoint}
  $G$ contains no two disjoint unbalanced circuits.
\end{claim}
\noindent{\it Proof of Claim \ref{cl: 2-disjoint}.}
Suppose to the contrary that $C_1$ and $C_2$ are two disjoint unbalanced circuits of $G$.  By Claim~\ref{cl:3-vertices},  $V_3(G)\setminus V(C_1\cup C_2)\neq \emptyset$.

Let $x\in V_3(G)\setminus V(C_1\cup C_2)$. By  Claim~\ref{cl: remove}-(A) and Observation \ref{ob: no theta}, there exists a $2$-edge-cut of $G$ separating $x$ from $V(C_1\cup C_2)$. Let $\{e_1,e_2\}$ be such a $2$-edge-cut.
Let
$$
\mathcal{F}=\{e_1\}\cup \{e\in E(G) : \{e,e_1\} \mbox{ is a $2$-edge-cut of $G$}\}
$$
and $\mathcal{B}$ be the set of all nontrivial components of $G-\mathcal{F}$.  Note that every member of $\cal B$ is $2$-connected. Since $d_G(x) = 3$, there is a   $B_0\in \mathcal{B}$  containing $x$.   Obviously $B_0$ doesn't contain $C_1$ or $C_2$, so $|\mathcal{B}| \geq 2$.

Let $B \in \mathcal{B}$. Then $|\delta_G(B)| = 2$.  If $B$ is balanced, then by (S3), $B$ contains at least two $2$-vertices and thus contains a circuit containing at least two $2$-vertices which  is removable.  If $B$ is unbalanced, then $B$ contains an unbalanced circuit which is also  is removable. Thus  each $B \in \mathcal{B}$ contains a removable circuits. Since $|\mathcal{B}| \geq 2$, by  Claim~\ref{cl: remove}-(A), $B$ is an unbalanced circuit  if it is unbalanced. Therefore every $B\in \mathcal{B}$ is either balanced or is an unbalanced circuit.  In particular, $C_1$ and $C_2$ are two distinct members   of $ \mathcal{B}$ and $|\mathcal{B}| \geq 3$.

Since $G$ is $2$-connected, there is a circuit that contains all edges in  $\mathcal{F}$ and goes through every $B \in \mathcal{B}$. We choose such a circuit $C$  with the following properties:

(1) $\sigma(C)  = \epsilon$ (the existence of $C$ is guaranteed since $C_1$ is unbalanced);

(2) subject to (1), $|V_2(G)\cap V(C -V(C_1))|$ is as large as possible;

(3)  subject to (1) and (2),  $|E_N(G)\cap E(C -V(C_1))|$ is as small as possible.

Since each $B$ is either balanced or is an unbalanced circuit, $G- V(C)$ is  balanced. Since $\sigma(C)  = \epsilon$,
by  Claim~\ref{cl: remove}-(B), $C$ is not removable and thus $C$ is balanced.

Let $B \in \mathcal{B}\setminus \{C_1\}$. If  $B$ is balanced or is unbalanced but not a circuit of length $2$,  then it contains a  $2$-vertex. Thus by (2)  $C$ contains at least one $2$-vertex  in $B$. If $B$ is an unbalanced circuit of length $2$, then  by (3), $C$  contains the positive edge in $B$.  In this case, since $C$ is balanced,  the other  edge in $B$ (which is negative) belongs to $\mathcal{U}(C)$.  Therefore every $B \in \mathcal{B}\setminus \{C_1\}$ contributes at least $1$ to
$|\mathcal{U}(C)|+|V_2(G)\cap V(C)|$. Since $|\mathcal{B}\setminus \{C_1\}| \geq 2$, we have $|\mathcal{U}(C)|+|V_2(G)\cap V(C)|\geq 2$. Hence $C$ is a removable circuit, a contradiction.
\hfill $\Box$

 \begin{claim}
\label{cl: theta}
  $G$ contains an unbalanced theta and $\epsilon=1$.
\end{claim}

\noindent{\it Proof of Claim \ref{cl: theta}.}
We first show that $G$ contains an unbalanced theta. Suppose that $G$ contains no unbalanced theta.   If $G$ is unbalanced, $G$ contains an unbalanced circuit. If
$G$ is balanced,  $|V_2(G)|\geq 4$ by (S3)
 and thus it has a circuit  containing at least two $2$-vertices since $G$ is $2$-connected.
Hence $G$ has a removable circuit $C$ in either case.  By the minimality of $G$, $G -V(C)$ has an NZW and by Lemma~\ref{lm: remove circuit}, we may extend this to a desired NZW of $G$, a contradiction. Therefore $G$ contains an unbalanced theta.

The existence of unbalanced thetas implies that $\epsilon \in \{-1,1\}$. Let $C$ be an unbalanced circuit.  By Claim \ref{cl: 2-disjoint}, $G$ contains no two disjoint unbalanced circuits, and thus $G-V(C)$ is balanced.  By   Claim \ref{cl: remove}-(B),  $\epsilon \neq \sigma(C)=-1$, so $\epsilon=1$.
 \hfill $\Box$

\begin{claim}\label{cl: k=1}
$|E_N(G)|\ge 2$.
\end{claim}

\noindent{\it Proof of Claim \ref{cl: k=1}.} By Claim~\ref{cl: theta}, $G$ is unbalanced.
Suppose to the contrary that $E_N(G)=\{e_0\}$. Let $P$ be the maximal  subdivided edge of $G$  containing $e_0$. Let $y_0, y_1$ be the two ends of $P$.  Then $Int(P)\subseteq V_2(G)$ and $y_0, y_1\in V_3(G)$. Let $G'=G-Int(P)$ if $Int(P) \not = \emptyset$; Otherwise, let $G' = G - e_0$.

We claim that $G'$ is $2$-connected. Otherwise, let $B$ be the maximal $2$-connected subgraph of $G'$ containing $y_1$. Then $B\neq G'$ and  $B$ is nontrivial since $d_G(y_1) = 3$. By the maximality of $B$, $\delta_{G'}(V(B))\neq \emptyset$ in which each edge is a bridge of $G'$. Thus $y_0\in V(G-V(B))$. Since $G$ is $2$-connected by Claim~\ref{cl: 2-connect},  $\delta_G(V(B))$ is a $2$-edge-cut of $G$. Note that $B$ is balanced and $G-V(B)$ is balanced and contains circuits since $y_0\in V_3(G)$. By  Claim \ref{cl: 2-edge-cut}, $V(B)\subseteq V_2(G)$, which contradicts  the fact $y_1\in V_3(G)$.

\medskip
\noindent
{\it (i)   $G'$ contains no circuit $C$ with  $V(C) \cap \{y_0,y_1\}  \not = \emptyset$ and  $|V(C)\cap V_2(G)| \geq 2$.}
\medskip

\noindent{\it Proof of (i).} Otherwise, $C$ is a removable circuit such that $G-V(C)$ is balanced and $\sigma(C) = 1 =\epsilon$, a contradiction to Claim \ref{cl: remove}-(B).

\medskip

Since $G'$ is a balanced shrubbery, 
 $|V_2(G')|\geq 4$ by (S3) and thus at least two of them,  say $y_2$ and $y_3$,  also belong to $V_2(G)$. Note  $\{y_2,y_3\}\cap \{y_0,y_1\} =\emptyset$.  By $(i)$, there is no circuit in $G'$ containing $\{y_1,y_2,y_3\}$. Thus by Theorem \ref{th: 1-negative}, there is a partition of $V(G')$ into $\mathcal{I}=\{X_1, X_2, Y_1, Y_2, Y_3\}$ such that $y_i\in Y_i$ ($i=1,2,3$),  $\delta_{G'}(X_1, X_2)=\delta_{G'}(Y_i, Y_j)=\emptyset$ ($1\leq i < j \leq 3$), and $\delta_{G'}(X_i, Y_j)=e_{ij}$ ($i=1,2$; $j=1,2,3$).  For each $Z\in \mathcal{I}$,
since $G'$ is $2$-connected and $|\delta_{G'}(Z)|\leq 3$,  $G'[Z]$ is connected.

 Since  $G'$ is $2$-connected and $|\delta_{G'}(Y_j)|= 2$ for $j \in \{2,3\}$, we have the following statement.

\medskip
\noindent
{\it (ii) For any $\{i,j\} = \{2,3\}$, there is a circuit $C_i$ in $G' - Y_j$ containing $y_1$ and all the edges in $\{e_{11}, e_{1i}, e_{2i}, e_{21}\}$.  We choose $C_i$ such that $|V(C_i)\cap V_2(G)|$  is as large as possible.    Then by (i),  $|V(C_i)\cap V_2(G)| \leq 1$.}

\medskip
\noindent
{\it (iii)  $y_0 \not \in Y_2\cup Y_3$, $Y_2=\{y_2\}$, and $Y_3=\{y_3\}$.}
\medskip

\noindent{\it Proof of (iii).}  Let $j \in \{2,3\}$. We first show $|Y_j| = 1$ if  $y_0\notin Y_j$.   WLOG suppose to the contrary $y_0 \not \in Y_3$ and $|Y_3| \geq 2$.   Since $y_0 \not \in Y_3$, $|\delta_G(Y_3)|= 2$.  By (ii)  $C_2$ is a circuit in   $G'-Y_3$. Since $G'[Z]$ is connected for each $Z \in \mathcal{I}$, $G'-Y_3$ is connected. Thus  there is a $(y_0, C_2)$-path $P'$ in $G'-Y_3$, so $P'\cup P\cup C_2$ is an unbalanced theta in $G - Y_3$.  By  Claim \ref{cl: 2-edge-cut}, $Y_3\subseteq V_2(G)$. Thus $G[Y_3]$ is a path and  $Y_3\subset V(C_3)$. By the choice of $C_3$, $V(C_3)$ contains at most one $2$-vertex. This implies $|Y_3| = 1$.

Now we show $y_0 \not \in Y_2\cup Y_3$. Otherwise WLOG, assume $y_0 \not \in Y_3$ and $y_0\in Y_2$.
Then $Y_3 = \{y_3\}$ and $y_3 \in V_2(G)$.   By (S4), $C_3$ is not a  balanced $4$-circuit, and thus there is a set $Z\in \{Y_1, X_1, X_2\}$ such that $|V(C_3)\cap Z|\ge 2$. Since $|\delta_G(Z)|=3$, by Claim \ref{cl: 3-edge-cut}, $(V(C_3)\cap V_2(G))\cap Z\neq \emptyset$. Thus $|V(C_3)\cap V_2(G)|\ge |(V(C_3)\cap V_2(G))\cap Z|+|\{y_3\}|\ge 2$, a contradiction to $(ii)$. This shows $y_0 \not \in Y_2\cup Y_3$ and thus $|Y_2| = |Y_3| = 1$.

\medskip
\noindent
{\it (iv) $|X_i|=1$ if $y_0\notin X_i$ for any $i\in \{1,2\}$ and thus $ y_0 \in X_1\cup X_2$. }
\medskip

\noindent{\it Proof of (iv).}
Suppose that for some $i\in \{1,2\}$  $y_0\notin X_i$ and $|X_i| \ge 2$.  WLOG assume $i =1$. Let $x_{1j}$ be the end of $e_{1j}$ in $X_1$ for $j=1,2,3$.  Since $\Delta(G)\leq 3$,  $x_{11}\neq x_{1j}$ for some $j \in \{2,3\}$. Note that $x_{11}, x_{1j}\in V(C_j)$. Since $|\delta_G(X_1)|= 3$ and $G[X_1]$ is balanced, by Claim~\ref{cl: 3-edge-cut}, $(V(C_j)\cap V_2(G))\cap X_1\neq \emptyset$  by the choice of $C_j$. Since $y_j\in V(C_j)\cap V_2(G)$ by $(iii)$, $|V(C_3)\cap V_2(G)|\ge |(V(C_3)\cap V_2(G))\cap X_1|+|\{y_j\}|\ge 2$, a contradiction to $(ii)$.

If $y_0 \not \in X_1\cup X_2$, then $|X_1| = |X_2| = 1$. By (iii), $G[Y_2\cup Y_3\cup X_1\cup X_2]$ is a balanced $4$-circuit, a contradiction to (S4). Therefore $y_0 \in X_1\cup X_2$.

\medskip
By (iv), WLOG assume $y_0 \in X_1$. Then by (iv) and (iii), $|X_2| = |Y_2| = |Y_3| = 1$. Denote $X_2=\{x_2\}$.

\medskip
\noindent
{\it (v) $Y_1=\{y_1\}$.}
\medskip

\noindent{\it Proof of (v).}
Suppose to the contrary that $Y_1\neq \{y_1\}$. Then $|Y_1|\ge 2$ and $G[Y_1]$ is balanced. Let $C_4$ be a circuit containing all the edges in $\{e_{11}, e_{12}, e_{22}, e_{21}\}$ and $|V(C_4)\cap V_2(G)|$ is as large as possible.
Since $G[Y_1]$ is balanced and $|\delta_G(Y_1)|= 3$, by  Claim \ref{cl: 3-edge-cut}, $V(C_4)\cap Y_1\cap V_2(G) \not = \emptyset$.  Since $y_2 \in V(C_4)$,  $|V(C_4)\cap V_2(G)| \geq 2$. Since $\delta_G(Y_1) \cap C(V) = \{e_{11}, e_{21}\}$ and $|\delta_G(Y_1)| = 3$, $G- V(C_4)$ is balanced. Thus $C_4$ is a removable circuit, a contradiction to Claim \ref{cl: remove}-(B).
This completes the proof of $(v)$.

\medskip

Now we can complete the proof of the claim.
Let $x_{11}$, $x_{12}$ and $x_{13}$ be the ends of $e_{11}$, $e_{12}$ and $e_{13}$ in $X_1$, respectively. By (S4), $G[\{x_{12},x_{13},x_2,y_2,y_3\}]$ is not a $4$-circuit,  so $x_{12}\neq x_{13}$.

If $G'[X_1]$ contains two internally disjoint $(y_0,x_{12})$-path and $(y_0,x_{13})$-path, then $G'$ has a circuit $C_5$ which contains all vertices in $\{y_0, x_{12}, y_2, x_2, y_3, x_{13}\}$ and $\{y_2,y_3\}\subset V(C_5)\cap V_2(G)$, a contradiction to (i).  Hence $G[X_1]$ has a cut-edge separating $y_0$ from $\{x_{12},x_{13}\}$.

Let $B_1$  be the maximal $2$-connected subgraphs in $G[X_1]$ containing $y_0$. Then every edge in $\delta_{G[X_1]}(B_1)$ is a cut-edge of $G[X_1]$ by the maximality of $B_1$.   Since each $\delta_{G[X_1]}(B_1)$ is a cut-edge of $G[X_1]$ and   since $G'$ is $2$-connected and $|\delta_{G'}(X_1)|=3$, $|\delta_{G[X_1]}(B_1)|=1$ or $2$.
Since $G[X_1]$ has a cut-edge separating $y_0$ from $\{x_{12},x_{13}\}$, $x_{12}$ and $x_{13}$ are in the same component of $G[X_1]-B_1$. Denote this component by $B_2$.

 Let $P'$ be an $(x_{12}, x_{13})$-path in $G[B_2]$. Then $C_6 = P'\cup x_{12}y_2x_2y_3x_{13}$ is a balanced circuit containing at least two $2$-vertices in $G$ ($y_2$ and $y_3$) and thus $C_6$ is a removable circuit of $G$.   If $B_1$ has a circuit $C'$ containing $y_0$, then there  is  $(y_1,C')$-path $P''$ in $G'-V(C_6)$.  Recall that $P$ is the maximal subdivided edge in $G$ containing the only negative edge $e_0$. Thus $P\cup P''\cup C'$ is an unbalanced theta in $G'- V(C_6)$, a contradiction to Claim \ref{cl: remove}-(A). This implies $B_1$ is trivial and $V(B_1) = \{y_0\}$.

Let $z$ be the neighbor of $y_0$ in $B_2$. Then $ \delta_{G[X_1]}(B_2) = \{y_0z, e_{12}, e_{13}\}$. Since $x_{12}\neq x_{13}$,  $z\neq x_{1j}$ for some $j\in \{2,3\}$. Since $|\delta_G(B_2)|=3$, by Claim \ref{cl: 3-edge-cut}, $G[B_2]$ has a $(z,x_{1j})$-path containing at least one vertex in $V_2(G)$. Note $G[B_2] = G'[B_2]$. Thus $G'$ has a circuit containing  $y_0$ and at least two vertices in $V_2(G)$, a contradiction to (i). This completes the proof of Claim \ref{cl: k=1}.
\hfill $\Box$

\medskip

By Claim \ref{cl: k=1}, $\epsilon(G)=|E_N(G)|\ge 2$. Denote $\epsilon(G)=k$.
By Claim \ref{cl: 2-connect} and Theorem \ref{thm: linkage}, we can choose a minimum subset $S\subseteq E(G)\setminus E_N(G)$ such that $H=G/S$ satisfies the following properties:
\begin{itemize}
\item[(i)] $\Delta(H)\leq 3$;
\item[(ii)] $H-N(H)-\cup_{e\in N(H)} Int(P_e)$ is a $2$-connected planar graph with a facial circuit $C$, where $P_e$ is the maximal subdivided edge in $H$ containing $e$;
\item[(iii)] $x_1, \dots ,x_k, x_{k+1}, \dots , x_{2k}$ are pairwise distinct and lie in that cyclic order on $C$, where  $E_N(H)=E_N(G)=\{e_1, \dots, e_k\}$ and  $x_i, x_{k+i}$ are the two ends of $P_{e_i}$ for each $i\in [1,k]$.
\end{itemize}

For each $v\in V(H)$, let $G_v$  denote the corresponding component of $G-E(H)$. Clearly, $G_v$ is $2$-connected by the minimality of $S$. Moreover, $S=\cup_{v\in V(H)}E(G_v)$ and $E(G)=E(H)\cup S.$

\begin{claim}
\label{cl: k=2}
$k=2$ and $|Int(P_{e_1})|+|Int(P_{e_2})|=1$.
\end{claim}

\noindent{\it Proof of Claim \ref{cl: k=2}.} Since $k \geq 2$, it is easy to see by Claim \ref{cl: 2-edge-cut} and by the minimality of $S$ that  if $d_H(x) = 2$ then $G_x= \{x\}$.
We first construct a circuit $C_H$ in the following cases. If  there are distinct $i,j\in [1,k]$ such that $|Int(P_{e_i})|=|Int(P_{e_j})|=0$, let $C_H=C$; If  $|Int(P_{e_i})|+|Int(P_{e_{i+1}})|\ge 2$ for some $i\in [1,k]$, let $C_H=C - E(x_iCx_{i+1}) -E(x_{i+k}Cx_{i+k+1})+P_{e_i} + P_{e_{i+1}}$. Note that $G_v$ is $2$-connected for any $v\in V(H)$, $\Delta(H)\leq 3$ and $\Delta(G)\leq 3$. Then $C_H$ can be extended to a removable circuit $C_G$ of $G$ and $G-V(C_G)$ is also balanced, a contradiction to  Claim \ref{cl: remove}-(B). So the claim holds.
\hfill $\Box$

\medskip

WLOG assume that $Int(P_{e_1})=\emptyset$ and $Int(P_{e_2})=\{y\}$ by Claim \ref{cl: k=2}. Then $P_{e_1}=x_1x_3$ and $P_{e_2}=x_2yx_4$. Denote $A_i=x_i C x_{i+1}$ $\pmod 4$ for $i\in [1,4]$,  ${C}_1=P_{e_1}\cup A_1\cup P_{e_2}\cup A_3$, and ${C}_2=P_{e_1}\cup A_4\cup P_{e_2}\cup A_2$. Note that both $C_1$ and $C_2$ contain the $2$-vertex $y$.

\begin{claim}
\label{cl: H=G}
$H=G$ and $V_2(G)=\{y\}$.
\end{claim}

\noindent{\it Proof of Claim \ref{cl: H=G}.}
As noted in the proof of Claim \ref{cl: k=2}, $G_y = \{y\}$. Let $x \in V(C)$. WLOG assume $x \in V(C_1)$.  Suppose that  $G_x$ is nontrivial. Then $G_x$ is balanced and  $|\delta_G(G_x)|\leq 3$. Since $G_x$ is $2$-connected, by Claim \ref{cl: 3-edge-cut}, ${C}_1$ can be extended to a circuit $C$  of $G$  such that $C$ contains the $2$-vertex $y$ and one $2$-vertex in $G_x$. Thus $C$ is balanced and  removable and $G-V(C)$ is balanced, a contradiction to  Claim \ref{cl: remove}-(B). Hence, $G_x$ is trivial.

  Assume that  there exists a vertex $u\in (V(G)\setminus (V(C)\cup \{y\})) \cap V_2(G)$. Since $G$ is $2$-connected, there are two  internal disjoint $(u,C)$-paths  $Q_1$ and $Q_2$ with $v_1$ and $v_2$ the end vertices in $C$ respectively.  Since $\Delta(G) \leq 3$, $v_1\not = v_2$. Let $C_3=Q_1 \cup Q_2 \cup v_1 C v_2$ and $C_4 \in \{{C}_1, {C}_2\}$ such that $V(C_4)\cap \{v_1, v_2\}
  \neq \emptyset$. Then $C'=C_3 \Delta C_4$ is a circuit  containing two $2$-vertices and the two negative edges. Thus $C$ is balanced and removable and  $G-V(C')$ is balanced, which contradicts  Claim~\ref{cl: remove}-(B). Thus $V_2(G)=\{y\}$.

Let $x$ be a $3$-vertex in $ V(H)\setminus V(C)$. If $G_x$ is nontrivial, then $G_x$ is balanced and $|\delta_G(G_x)| = 3$. By (S3),
  $G_x$ contains a $2$-vertex, a contradiction to the fact that $y$ is the only $2$-vertex in $G$.  Thus $G_x$ is trivial and  therefore $H=G$. \hfill $\Box$

\begin{claim}
\label{cl: segment}
 $Int(A_i)\neq \emptyset$ for each $i\in [1,4]$.
\end{claim}

\noindent{\it Proof of Claim \ref{cl: segment}.}
Suppose to the contrary that there is some $i\in [1,4]$, say $i=1$, such that $Int(A_1)= \emptyset$. Then $A_1$ is a chord  in $\mathcal{U}(C_2)$.  Since $C_2$ contains the $2$-vertex $y$,  ${C}_2$ is a removable circuit of $G$, a contradiction to Claim \ref{cl: remove}-(B) since $G-V(C_2)$ is balanced.
\hfill$\Box$

\medskip
\noindent
\textbf{The final step.}

By Claim~\ref{cl: segment}, let $y_1\in Int(A_1)$ be the neighbor of $x_1$. Let $Q$ be the component of $G-E(C)$ containing $y_1$. Since $d_G(y_1)=3$ by Claim~\ref{cl: H=G}, $Q$ is nontrivial. Obviously, $V(Q)\cap \{x_1,x_2,x_3,x_4\}=\emptyset$ since $\Delta(G)=3$.

  If there is a vertex $y_2$ in $V(Q)\cap (Int(A_2)\cup Int(A_3))$,  let $P$ be a $(y_1,y_2)$-path in $Q$. Since $\Delta(G)\leq 3$, $C_3=P\cup y_1Cy_2$ is a circuit containing $x_2$. Then
$C'={C}_2 \bigtriangleup C_3$ is a circuit of $G$ containing $y$ and  the chord $x_1y_1 \in \mathcal{U}(C')$. Thus $C'$ is a removable circuit of $G$,  a contradiction to   Claim~\ref{cl: remove}-(B)  since $G-V(C')$ is balanced.

If $V(Q)\cap (Int(A_2)\cup Int(A_3))=\emptyset$, then $V(Q)\cap V(C)\subseteq Int(A_4)\cup Int(A_1)$. Note that $|V(Q)\cap V(C)|\ge 2$ since $G$ is $2$-connected. Let $y_2, y_3\in V(Q)\cap V(C)$ be two ends of a segment $P'$ of $A_4\cup A_1$ such that the length of $P'$ is as large as possible. By Claim \ref{cl: H=G}, $G'=G-x_1x_3-y$ is a $2$-connected planar graph with a facial circuit $C$, and so $T'=\delta_{G'}(V(P'))\cap E(C)$ is a $2$-edge-cut of $G'$. Let $T=T'$ if $y_2, y_3\in Int(A_1)$, and otherwise $T=T'\cup \{x_1x_3\}$. Then $T$ is an edge-cut of $G$  with $|T| \leq 3$ and   the component of $G-T$  containing $y_2$ is balanced and doesn't contain $y$.  By (S3),
 this component contains a $2$-vertex (distinct from $y$), which contradicts  $V_2(G)=\{y\}$ by Claim \ref{cl: H=G}.
This completes the  proof of Lemma \ref{lm: water}.
 \end{proof}


\end{document}